\documentclass{amsart}
\usepackage{graphicx} 
   \usepackage{amsfonts,amssymb,amsmath,amsthm}
    \usepackage{latexsym,mathrsfs,color,hyperref}

\begin{document}

\title{Anisotropic Improved Leray-Trudinger Inequality}

\author{Giuseppina di Blasio}
\author{Giovanni Pisante}
\author{Georgios Psaradakis}

\address[G. di Blasio and G. Pisante]{Dipartimento di Matematica e Fisica, Universit\`a degli Studi della Campania ``L. Vanvitelli", Viale Lincoln 5, 81100 Caserta, Italy}
\email{giuseppina.diblasio@unicampania.it}
\email{giovanni.pisante@unicampania.it}
\address[G. Psaradakis]{Department of Mathematics, University of Western Macedonia, 52100 Kastoria, Greece}
\email{gpsaradakis@uowm.gr}

\maketitle

\begin{abstract}

We establish a Leray–Trudinger type inequality in the anisotropic setting induced by a strongly convex Finsler norm \( F \). The result generalizes classical exponential integrability inequalities for Sobolev functions to the framework of anisotropic Sobolev spaces \( W_0^{1,n}(\Omega) \), where the standard Euclidean norm is replaced by \( F \) and the associated polar norm \( F^\circ \).  Moreover, in the class of anisotropically radial functions, we obtain the optimal exponential constant in the spirit of Moser’s sharp inequality. 

\vspace{0.5cm}


\end{abstract}

\date{}

\newcommand{\R}{\mathbb{R}}
 \newcommand{\Rn}{\mathbb{R}^n}
  \newcommand{\N}{\mathbb{N}}
   \newcommand{\Om}{\Omega}
    \newcommand{\test}{C_c^1}
     \newcommand{\prt}{\partial}
      \newcommand{\omn}{\omega_{n}}
       \newcommand{\dd}{\mathrm{d}}
        \newcommand{\Ln}{\mathcal{L}^n}
         \newcommand{\Hnmo}{\mathcal{H}^{n-1}}
          \newcommand{\e}{\varepsilon}
           \newcommand{\gl}{\lambda}
            \newcommand{\om}{\omega}
             \newcommand{\Div}{\operatorname{div}}
              \newcommand{\sgn}{\mathrm{sgn}}
               \newcommand{\vs}{\vec{\sigma}}
                \newcommand{\vphi}{\vec{\varphi}}
                 \newcommand{\sprt}{\mathop{\rm sprt}}
                  \newcommand{\dist}{\mathop{\rm dist}}
                   \newcommand{\esssup}{\operatorname{esssup}}
                    \newcommand{\ub}{{\mathbb{B}^n}}
                     \newcommand{\Ommo}{\Om\setminus\{0\}}
                      \newcommand{\smxoR}{(\frac{|x|}{R_\Om})}
                       \newcommand{\wtest}{W_0^{1,2}(\Om)}
                        \newcommand{\mxoR}{\Big(\frac{|x|}{R_\Om}\Big)}
                         \newcommand{\Ltwo}{\mathcal{L}^2}
                          \newcommand{\dS}{\dd\sigma}
                           \newcommand{\us}{{\mathbb{S}^{n-1}}}
                             \newcommand{\uw}{{\mathcal{W}}}
\def\aver{\mathop{\int\mkern-18.5mu{-}}\nolimits}
\newtheorem{theorem}{Theorem}[section]
 \newtheorem{proposition}[theorem]{Proposition}
  \newtheorem{lemma}[theorem]{Lemma}
   \newtheorem{corollary}[theorem]{Corollary}
     \newtheorem{definition}[theorem]{Definition}
      \newtheorem{example}[theorem]{Example}
       \newtheorem{remark}[theorem]{Remark}
        \newtheorem{assumption}[theorem]{Assumption}
         \newtheorem{acknowledgements}{Acknowledgements}

\section{Introduction}

The classical Hardy and Trudinger-Moser inequalities represent fundamental tools in the analysis of Sobolev spaces, particularly in the borderline case \( p = n \), where the Sobolev embedding into \( L^\infty \) fails. In this critical regime it is well known that the optimal substitute is given by exponential-type integrability estimates, such as those introduced by Trudinger \cite{Tr}, later refined by Moser \cite{M} and further improved in a variety of contexts including singular potentials and weights (see for example \cite{AdSand}, \cite{CR}, \cite{CR2}, \cite{dOdO}, \cite{INW} and \cite{BRuf}).

In this framework, the so-called Leray–Trudinger inequalities emerge as an interpolation between the Hardy-type singular potential structure and the exponential summability of Sobolev functions. To be more precise, let $\Omega$ be a bounded domain of $\mathbb{R}^n$, $n\in\mathbb{N}\setminus\{1\}$, and set $R_\Omega:=\sup_{x\in\Omega}|x|$. Consider also the auxiliary function $X_1(t):=(1-\log t)^{-1}$, $t\in(0,1]$. Psaradakis and Spector in \cite{PsSp} introduced a sharp inequality for the Leray difference functional,
\[
I_n[u; \Omega] := \int_{\Omega} |\nabla u|^n \, dx - \left(\frac{n-1}{n}\right)^n \int_{\Omega} \frac{|u|^n}{|x|^n}X_1^n\left(\frac{|x|}{R_\Omega}\right) \, dx,
\]
showing that for any \( u \in W_0^{1,n}(\Omega) \) with $I_n[u; \Omega]\leq 1$, one can recover exponential integrability of the form
\[
\int_{\Omega} \exp\left( A_{n,\varepsilon} \left[ |u(x)| X_{1}^{\varepsilon}\left(\frac{|x|}{R_\Omega}\right) \right]^{\frac{n}{n-1}} \right) dx < +\infty,
\]
for every \( \varepsilon > 0 \), with a positive constant $A_{n,\varepsilon}$ depending only on $n$ and $\varepsilon$.

This result was later refined by Mallick and Tintarev \cite{MT}, who replaced the weight function \( X_1 \) with a doubly logarithmic correction \( X_2(t) := X_{1}(X_{1}(t)) \), obtaining a slightly better form of the inequality for a class of functions normalized by the Hardy difference. Later on in \cite{BlPPs-ind}, building on these results, we obtained the sharp version of the inequality.  

Parallel to these developments, a significant research effort has been devoted to generalizing these inequalities to the anisotropic setting. The exploration of anisotropic problems originates from G. Wulff’s 1901 research on crystal morphology and the minimization of anisotropic surface tensions. This area has grown in significance across various disciplines, notably in examining phase transitions and phase separation within multiphase materials (see for example \cite{AIM,BP} and \cite{LEP}). Hence, it becomes natural to adapt classical tools, traditionally effective for solving variational problems, to the Finsler framework (see for example \cite{BlL}, \cite{BlPPs}, \cite{BNP1}, \cite{BNP2}, \cite{Moll}, \cite{zouzou} and \cite{WaX}). 
In this framework, the Euclidean norm is replaced by a convex, $1$-homogeneous function \( F \colon \mathbb{R}^n \to [0,+\infty) \), with associated polar norm \( F^\circ \), and the corresponding anisotropic Sobolev spaces are built upon the gradient norm \( F(\nabla u) \). Anisotropic Hardy-type inequalities have been studied for instance in \cite{PBG,BlPPs,MST}. In particular, in \cite{MST} the sharp version of the  Hardy inequality, where the classical potentials \( |x|^{-p} \) are replaced by \( \left(F^\circ(x)\right)^{-p} \), has been proved and the best constants are explicitly characterized in terms of the geometry of the so-called Wulff shape.

The aim of the present work is to extend the Leray-Trudinger inequality to this anisotropic setting by proving the exponential integrability of functions in \( W^{1,n}_0(\Omega) \) under the control of anisotropic Hardy-type differences, and to obtain the optimal exponent in the case of anisotropically radial functions. More precisely, we prove that for any strongly convex Finsler norm \( F \), there exists a constant \( \gamma > 0 \) such that
\[
\int_\Omega \exp\left( \gamma \left[ |u(x)| X_2\left( \frac{F^\circ(x)}{R_\Omega} \right) \right]^{\frac{n}{n-1}} \right) dx < +\infty,
\]
for every \( u \in W_0^{1,n}(\Omega) \), where \( R_\Omega := \sup_{x \in \Omega} F^\circ(x) \) (see Section \ref{main-section}). Furthermore, for anisotropically radial functions, we establish a refinement of this inequality with a sharp value of the exponential constant, recovering the optimal growth rate in the spirit of Moser in the Finsler setting.

Our proofs combine techniques from anisotropic calculus of variations (including Hardy difference estimates, divergence structures, and vectorial inequalities) with integral representation formulae inspired by \cite{GT}, and follow a strategy similar in spirit to the works of \cite{FT}, \cite{PsSp}, and \cite{MT}, adapted to the Finsler framework.

\section{Notation and Preliminary Results}

We assume throughout that $F:\Rn\rightarrow[0,\infty)$ is even and positively $1$-homogenous, which is also strongly convex in $\Rn\setminus\{0\}$; that is, $F\in C^{2}(\Rn\setminus\{0\})$ and $ \nabla^{2}F^{2}(\xi)$ is a positive definite matrix for any $\xi\not=0$. Defining the polar $F_{\circ}$ of $F$ we have the Cauchy-Schwarz inequality:
\begin{equation}\label{Cauchy Schwarz}
 |\xi\cdot x|\leq F(\xi)F_\circ(x),\qquad\xi,~x\in\Rn.
\end{equation}
Also, the assumed homogeneity implies (Euler's theorem)
\begin{equation}
	\label{F-prop-01}
	\nabla F_{\circ}(x)\cdot x = F_{\circ}(x),\qquad x\in\Rn\setminus\{0\}
\end{equation}
and
\begin{equation}
	\label{F-prop-02}
	\nabla^{2} F_{\circ}(x)\cdot x = 0 ,\qquad x\in\Rn\setminus\{0\}.
\end{equation}
Since, differentiating \eqref{F-prop-01}, we have
\[
\nabla (F_{\circ}) = \nabla (\nabla F_{\circ}(x)\cdot x) = \nabla^{2} F_{\circ}(x)\cdot x  + \nabla (F_{\circ}).
\]
We explicitly observe that $F$  and $ F_{\circ}$ are equivalent to the euclidean norm, i.e. there exist constants $\alpha, \beta, \alpha', \beta'$ such that for any $\xi \in \Rn$ it holds
\begin{equation}
\label{norm-equivalence}
\alpha | \xi | \leq F(\xi) \leq \beta  | \xi |  \;;\;\;\  \alpha'  | \xi | \leq F_{\circ}(\xi) \leq \beta'  | \xi |.
\end{equation}
Moreover by zero-homogeneity of $\nabla F_{\circ}$, we also have for any $\xi\in \R^{n}$ 
\begin{equation}
\label{bound-gradient-F0}
|\nabla F_{\circ}(\xi)| \leq \max_{x\in S^{n-1}} |\nabla F_{\circ}(x) | =: \tilde{\beta}
\end{equation}
Using \eqref{F-prop-01}, one easy sees that
\begin{equation}\label{divengence-theta}
 \Div\left(\frac{x}{F_{\circ}(x)}\right)=\frac{n-1}{F_\circ(x)},\qquad x\in\Rn\setminus\{0\},
\end{equation}
which in turn implies
\begin{equation}\label{divergence-zero}
 \Div\left(\frac{x}{F^n_{\circ}(x)}\right)=0,\qquad x\in\Rn\setminus\{0\}.
\end{equation}
From \eqref{divergence-zero}, applying \eqref{F-prop-02}, we get
\begin{equation}
\label{div-zero-F}
\Div \left( \frac{x}{F^n_{\circ}(x)} | \nabla F_{\circ}| \right) =0,
\end{equation}
since
\[
\begin{split}
\Div \left( \frac{x}{F^n_{\circ}(x)} | \nabla F_{\circ}| \right) &  =  \Div\left(\frac{x}{F^n_{\circ}(x)}\right) + \frac{x}{F^n_{\circ}(x)}\cdot \nabla |\nabla F_{\circ}(x)| \\
& =  \frac{x}{F^n_{\circ}(x)}\cdot \frac{\nabla F_{\circ}(x)}{|\nabla F_{\circ}(x)|} \otimes \nabla^{2} F_{\circ}(x). 
\end{split}
\]
The sublevel sets of $F$ and $F_{\circ}$
\[
B_{F}:=\{ x\in \R^{n} \,\:\, F(x) \leq 1  \}  \;\text{ and }\; B_{F_{\circ}}:=\{ x\in \R^{n} \,\:\, F_{\circ}(x) \leq 1  \} 
\]
are polar to each other and the \emph{Wulff ball}, denoted by $\mathcal{W}$, is defined as the interior of $B_{F_{\circ}}$. 
We will use the notation $\omega_{n,F}$ to denote the volume of $\mathcal{W}$. Recall that $\mathcal{H}^{n-1}(\partial \mathcal{W})	= n \omega_{n,F}$.

For a given domain $\Omega\subset \R^{n}$ we define the anisotropic inner-radius as  
\[
R_\Om:=\sup_{x\in\Om}F_{\circ}(x).
\] 

We will use the following auxiliary functions
\[X_1(t):=(1-\log t)^{-1},\qquad t\in(0,1], \quad X_1(0):=0.\] 
\[
 X_2(\cdot):=X_1\big(X_1(\cdot)\big).
\]
It is not difficult to verify that 
\begin{equation}
\label{log-prop-01}
-\log (X_{1})=  \frac{1-X_{2}} {X_{2}} \leq \frac{1}{X_{2}}
\end{equation}
and
\begin{equation}
\label{log-prop-02}
X_{2}\log (X_{1}^{-1})= 1-X_{2} =  \frac{\log(X_{1}^{-1})}{1+\log(X_{1}^{-1})}.
\end{equation}

Moreover we can compute
\begin{equation}
\label{derivata-X}
\frac{d}{dt}X_{1}(t)= \frac{1}{t} X_{1}^{2} \;,\;\;\; \frac{d}{dt}X_{2}(t)= \frac{1}{t} X_{1}X_{2}^{2}.
\end{equation}

We recall an elementary lemma (see \cite[Lemma 2.5]{BlPPs-ind} and \cite[Lemma 6.1]{Ps}).

\begin{lemma}\label{lemma onedim} For any
$g\in AC\big([0,1]\big)$ with $g(1)=0$ the following estimates hold
\begin{align}\label{onedimlemma stepone INEQ assist}
 |g(r)|
  \leq
   \Big(-\log X_1(r)\Big)^{1/2}
    \Bigg(\int_0^1t|g'(t)|^2X_1^{-1}(t)~\dd t\Bigg)^{1/2},
\end{align}

\begin{align}\label{onedimlemma stepone INEQ}
 \sup_{r\in[0,1]}\Big\{|g(r)|X_2^{1/2}(r)\Big\}
  & \leq
   \Bigg(\int_0^1 t|g'(t)|^2X_1^{-1}(t)~\dd t\Bigg)^{1/2}.
\end{align}
 \end{lemma}

From now on, we systematically drop the argument $x$ from the expression $F_\circ(x)$. Furthermore, we use the simplified notation $X_1$, $X_2$ in place of $X_1(F_{\circ}(x)/R_\Om)$, $X_2(F_{\circ}(x)/R_\Om)$.

Given a positive measure $\mu$ on $\Rn$ we define the Hardy differences as follows:
\[
I_{F,\mu}[u]:=\int_\Om F(\nabla u)^n \dd \mu
  -\left(\frac{n-1}{n}\right)^n\int_\Om\frac{|u|^n}{F^n_{\circ}}X^n_1\left(\frac{F_{\circ}}{R_\Om}\right)~\dd \mu\,,
\]
\[
J_{F,\mu}[u]:=\int_\Om \Big|\nabla u\cdot\frac{x}{F_{\circ}}\Big|^n \dd \mu
  -\left(\frac{n-1}{n}\right)^n\int_\Om\frac{|u|^n}{F^n_{\circ}}X^n_1\left(\frac{F_{\circ}}{R_\Om}\right)~\dd \mu\,.
\]
When $\mu=|\nabla F_{\circ}|dx$, we will use the simplified notation $I_{F}[u]$ and $J_{F}[u]$. Note that since, by the homogeneity of $F_{\circ}$ and \eqref{Cauchy Schwarz}, we have
\[
\Big|\nabla u\cdot\frac{x}{F_{\circ}}\Big| \leq F(\nabla u)
\]
we easily deduce
\begin{equation}
\label{JvsI}
J_{F,\mu}[u] \leq I_{F,\mu}[u].
\end{equation}
Now we establish a couple of standard lower estimates for the Hardy-Leray difference.
\begin{proposition}\label{proposition:link}
Let $\Om\subset \R^{n}$ be a bounded domain. Consider a measure $\mu=g(x)\dd x$ absolutely continuous with respect to the Lebesgue measure, with density $g\in W^{1,\infty}(\Omega)$, satisfying the equality
\begin{equation}\label{Hp-g}
\int_\Om  \frac{x}{F_{\circ}} \cdot \nabla g \, \dd x = 0.
\end{equation}
 Given $u\in\test(\Om)$ and defined $v:=X_1^{1-1/n}u$, we have 
\begin{equation}\label{link}
 \int_\Om F_{\circ}^{2-n}|v|^{n-2}\Big|\nabla v \cdot \frac{x}{F_{\circ}}\Big|^2X_1^{-1}\left(\frac{F_{\circ}}{R_\Om}\right)~\dd \mu
  \leq
   \kappa_n J_{F,\mu}[u],
\end{equation}
where 
\begin{equation}
\label{cappan}
\kappa_n:=\frac{2}{n} \left(\frac{n}{n-1}\right)^{n-2}.
\end{equation}
Moreover there exist a constant $\sigma_F\geq1$, depending only on $F$, such that 
\begin{equation}\label{link2}
 \int_\Om F^n(\nabla v)X_1^{1-n}\left(\frac{F_{\circ}}{R_\Om}\right)~\dd \mu
  \leq
   2^{n-1}\sigma^n_F I_{F,\mu}[u].
\end{equation}
\end{proposition}

\begin{proof} It suffices to consider $u\in\test\big(\Om\setminus\{0\}\big)$. Setting $u = X_1^{-1+1/n}v$ we compute
\[
 \nabla u \cdot \frac{x}{F_{\circ}}
  =
   \underbrace{X_1^{-1+1/n}\nabla v\cdot\frac{x}{F_{\circ}}}_{b}
   -
   \underbrace{\frac{n-1}{n}X_1^{1/n}\frac{v}{F_{\circ}}}_{a}. 
\]
Applying the inequality
\begin{equation}\label{vec}
 |b-a|^n-|a|^n
  \geq
   \frac{n}{2}|a|^{n-2}|b|^2-n|a|^{n-2}a\cdot b,
\end{equation}
since we have: 
\begin{align*}
 |b-a|^n-|a|^n & = \left|\nabla u\cdot\frac{x}{F_{\circ}}\right|^n
  -\left(\frac{n-1}{n}\right)^n\frac{|u|^n}{F_{\circ}^n}X^n_1, \\
|a|^{n-2}|b|^2 & =  \Big(\frac{n-1}{n}\Big)^{n-2} F_{\circ}^{2-n}|v|^{n-2}\Big|\nabla v \cdot \frac{x}{F_{\circ}}\Big|^2X_1^{-1}, \\
  n|a|^{n-2} a\cdot b & =
  \Big(\frac{n-1}{n}\Big)^{n-1}\nabla\big(|v|^n\big)\cdot\frac{x}{F^n_{\circ}},
\end{align*}
after an integration on $\Omega$, we arrive at \eqref{link}, observing that by hypothesis \eqref{Hp-g}, we get
\begin{align*}
n\int_\Om|a|^{n-2} a\cdot b~\dd \mu & =0.
\end{align*}

Now in order to show \eqref{link2}, we compute
\[
 \nabla u
  =
   \underbrace{X_1^{-1+1/n}\nabla v}_{\xi_2}
   -
   \underbrace{\frac{n-1}{n}X_1^{1/n}\frac{v}{F_{\circ}}\nabla_xF_\circ}_{\xi_1},
\]
and use the following vectorial inequality
\begin{equation}\nonumber
 F^n(\xi_2-\xi_1)
  - F^n(\xi_1)
   \geq\frac{2^{1-n}}{\sigma^n_F}F^n(\xi_2)
    - nF^{n-1}(\xi_1)F_{\xi}(\xi_1)\cdot\xi_2, \;\; \forall~\xi_1,\xi_2\in\Rn,
\end{equation}
that is the case with $p=n$ of \cite[\S.4 Lemma 2]{BlPPs}. 

Clearly,
\begin{align*}
  \int_\Om\Big(F^n(\xi_2-\xi_1)-F^n(\xi_1)\Big)~\dd 
  \mu & = I_{F,\mu}[u], \\
  \int_\Om F^n(\xi_2)~\dd \mu & = \int_\Om F^n(\nabla v)X_1^{-n+1}~\dd \mu.
\end{align*}
Since the $1$-homogeneity of $F$ implies 
\[
F_\xi(t\xi)=\sgn\{t\}F_\xi(\xi)\;,  \;\;t\neq0\,, \;\;\xi\in\Rn\setminus\{0\},
\]
we also have
\begin{align}\nonumber
 n\int_\Om F^{n-1}(\xi_1)F_{\xi}(\xi_1)\cdot\xi_2~\dd \mu
  & = \Big(\frac{n-1}{n}\Big)^{n-1}
   \int_\Om\nabla\big(|v|^n\big)\cdot\frac{F_{\xi}\big(\nabla_x F_{\circ})}{F^{n-1}_{\circ}}~\dd \mu.
\end{align}
Finally, we use the known fact that $F_{\xi}\big(\nabla_x F_{\circ})=x/F_{\circ}$ (see for example \cite[Lemma 2.2]{BP}) to get
\begin{align}\nonumber
  \int_\Om\nabla\big(|v|^n\big)\cdot\frac{F_{\xi}\big(\nabla_x F_{\circ})}{F^{n-1}_{\circ}}~\dd \mu
  = 
   \int_\Om\nabla\big(|v|^n\big)\cdot\frac{x}{F^n_{\circ}}~\dd \mu=0,
\end{align}
the last equality holds because of \eqref{Hp-g}. 
\end{proof}

\begin{remark}
We observe that due to relations \eqref{divergence-zero} and \eqref{div-zero-F}, the results of Proposition \ref{proposition:link}, holds in particular for $\mu$ being the Lebesgue measure, $\dd x$, and for $\mu=|\nabla F_{\circ}| \dd x$.
\end{remark}

\section{Main results}
\label{main-section}
We start with an auxiliary result that relies on the ideas of \cite[Thorem B and Proposition 3.2]{BFT1}, generalized in the anisotropic setting.

\begin{proposition}
    Let $\Omega$ be a bounded domain and  $\kappa_n$ given by \eqref{cappan}. For any $u\in W^{1,n}_0(\Omega)$, we have
      \begin{equation}
  \label{improved-01}
        \int_\Omega \frac{|u|^n}{F^n_\circ}X^n_1\left(\frac{F_{\circ}}{R_\Om}\right)X^2_2\left(\frac{F_{\circ}}{R_\Om}\right) dx\leq  n^{2}  \kappa_n I_{F,dx}[u].
    \end{equation}
\end{proposition}

\begin{proof}
  Let $v:=X_1^{1-1/n}u$, we can rewrite the left hand side of \eqref{improved-01} as
\begin{equation}
\label{trans-01}
 \int_\Omega \frac{|v|^n}{F^n_\circ}X_1X^2_2 dx = \int_{\Omega} |v|^{n} \Div \left( \frac{X_{2}}{F_{\circ}^{n}} \cdot x \right) dx, 
\end{equation}
indeed, using the properties of $X_{i}$ and \eqref{divergence-zero}, it follows that
\[
 \Div \left( \frac{X_{2}}{F_{\circ}^{n}(x)} \cdot x \right) =  \Div \left( \frac{x}{F_{\circ}^{n}(x)}\right) X_{2} +   \frac{x}{F_{\circ}^{n}(x)} \nabla X_{2} = \frac{X_{2}^{2}X_{1}}{F^n_\circ(x)}.
\]
From \eqref{trans-01}, using the divergence theorem and H\"older's inequality we have 
\[
\begin{split}
\int_\Omega \frac{|v|^n}{F^n_\circ(x)}X_1X^2_2 dx  & =   -n \int_{\Omega} |v|^{n-1} \nabla v \cdot  \frac{x}{F_{\circ}^{n}(x)}X_{2}\, dx \\
& \leq n \int_{\Omega} |v|^{n-1} \left|  \nabla v \cdot  \frac{x}{F_{\circ}(x)} \right| X_{2} F_{\circ}^{1-n}(x) \, dx \\
& \leq n \left(  \int_{\Omega} |v|^{n-2} \left|  \nabla v \cdot  \frac{x}{F_{\circ}(x)} \right| ^{2} X_{1}^{-1} F_{\circ}^{2-n}(x) \, dx \right)^{\frac{1}{2}} \left( \int_\Omega \frac{|v|^n}{F^n_\circ(x)}X_1X^2_2 dx  \right)^{\frac{1}{2}}.
 \end{split}
 \]
 Therefore, using \eqref{link} and \eqref{JvsI}, we obtain \[
 \int_\Omega \frac{|v|^n}{F^n_\circ(x)}X_1X^2_2 dx \leq n^{2}  \kappa_n J_{F,dx}[u] \leq n^{2}  \kappa_n I_{F,dx}[u]
\]
and the thesis follows from the definition of $v$ in terms of $u$.
\end{proof}

The following result is the key estimate for the proof of Theorem \ref{anis-main-theorem} for general functions, and generalizes to the anisotropic setting Proposition 3.1 of \cite{MT}.

\begin{proposition} \label{uX2estimate}
    Let $u\in W^{1,n}_0(\Omega)$, for any $q>n$, we have
    \begin{equation}
    \label{q-estimate-anisotropic}
    \left \| u X_2^{\frac{2}{n}}\left(\frac{F_{\circ}(x)}{R_\Om}\right)\right\|_{L^q(\Omega)} \leq C_{n,F} \left[ 1+\frac{q(n-1)}{n} \right]^{1-\frac{1}{n}+\frac{1}{q}} |\Omega|^{\frac{1}{q}}  \big(I_{F,dx}[u] \big)^\frac{1}{n}
    \end{equation}
with $C_{n,F}$ defined by
\begin{equation}
\label{cienne}
C_{n,F}:=\frac{1}{n} \omega_{n}^{-\frac{1}{n}}   \left[  \alpha^{-1} 2^\frac{n-1}{n}\sigma_F +  \tilde{\beta} \frac{n+1}{n}  n^\frac{2}{n}  \kappa_n^{\frac{1}{n}} \right],
\end{equation} 
where $\sigma_{F}$ and $\kappa_{n}$ are the constants in Proposition \ref{proposition:link}, $\alpha$ and $ \tilde{\beta}$ come from \eqref{norm-equivalence} and \eqref{bound-gradient-F0} and $\omega_{n}$ is the Lebesgue measure of the unit ball.
\end{proposition}

\begin{proof}
As usual we can assume w.l.o.g. that $u$ is a smooth positive function.  Set $v:=X_1^{1-1/n}u$ and using the classical integral representation (see for example \cite[Lemma 7.14]{GT}), we can write
\[
\begin{split}
 \left | u X_2^{\frac{2}{n}}\Big(\frac{F_{\circ}}{R_\Om}\Big)\right |  = &   \left | v X_{1}^{\frac{1}{n}-1}\Big(\frac{F_{\circ}}{R_\Om}\Big)X_2^{\frac{2}{n}}\Big(\frac{F_{\circ}}{R_\Om}\Big)\right |  \\
 = &  \left|  \frac{1}{n \omega_{n}} \int_{\Omega} \frac{(x-y)\cdot \nabla \left( v(y) X_{1}^{\frac{1}{n}-1}X_{2}^{\frac{2}{n}}  \right) }{|x-y|^{n}}dy  \right| \\
 \leq & \frac{1}{n \omega_{n}} \underbrace{ \int_{\Omega} \frac{\nabla v(y)}{|x-y|^{n-1}} X_{1}^{\frac{1}{n}-1} dy}_{K_{1}(x)} \\
 & +\ \frac{1}{n \omega_{n}} { \int_{\Omega} \frac{|v(y)|}{|x-y|^{n-1}}  \left| \nabla( X_{1}^{\frac{1}{n}-1} X_2^{\frac{2}{n}} ) \right| dy } \\
 \leq &  \frac{1}{n \omega_{n}} \left( K_{1} +   \frac{n+1}{n} \underbrace{ \int_{\Omega} \frac{|v(y)|}{|x-y|^{n-1}} X_{1}^{\frac{1}{n}} X_2^{\frac{2}{n}}  \frac{|\nabla F_{\circ}|}{F_{\circ}} dy }_{K_{2}(x)}  \right), 
\end{split}
\]
where we have used that $X_{2}\leq 1$, property \eqref{derivata-X} and the following inequality
\[
\begin{split}
 \left| \nabla( X_{1}^{\frac{1}{n}-1} X_2^{\frac{2}{n}} ) \right|  = &  \left| \left( \frac{1}{n}-1 \right)  X_{1}^{\frac{1}{n}-2} X_2^{\frac{2}{n}} \frac{\nabla F_{\circ}}{R_{\Omega}} X_{1}^{'} + \frac{2}{n}X_{1}^{\frac{1}{n}-1} X_2^{\frac{2}{n}-1}   \frac{\nabla F_{\circ}}{R_{\Omega}} X_{2}^{'} \right| \\
 = & \left| X_{1}^{\frac{1}{n}} X_2^{\frac{2}{n}+1}   \frac{\nabla F_{\circ}}{F_{\circ}} \left( \frac{1-n}{n} X_{2}^{-1}+\frac{2}{n} \right) \right| \\
 \leq &  X_{1}^{\frac{1}{n}} X_2^{\frac{2}{n}}  \frac{|\nabla F_{\circ}|}{F_{\circ}}  \frac{n+1}{n}\
\end{split}
\]
that is true since 
\[
X_{2} \left( \frac{1-n}{n} X_{2}^{-1}+\frac{2}{n} \right) \leq \frac{n+1}{n}. 
\]
Now we have
\begin{equation}
\label{finale-01}
 \left \| u X_2^{\frac{2}{n}}\right \|_{L^{q}(\Omega)}  \leq  \frac{1}{n \omega_{n}} \left( \|K_{1}\|_{L^{q}(\Omega)} +   \frac{n+1}{n} \| K_{2}\|_{L^{q}(\Omega)}  \right) 
\end{equation}

To estimate the previous norms we first observe that, for $q>n$, defined $r$ such that 
\begin{equation}
\label{indici-01}
\frac{1}{n}+\frac{1}{r}=1+\frac{1}{q}
\end{equation} 
and
\[
h_{r}(x):= \int_{\Omega} \frac{1}{|x-y|^{(n-1)r}} dy,
\]
we have (cfr. \cite[proof of Proposition 3.1]{MT})
\begin{equation}
\label{Norma-hr}
\| h_{r} \|_{L^{\infty}(\Omega)}^{\frac{1}{r}} \leq \omega_{n}^{1-\frac{1}{n}} \left( 1+\frac{q(n-1)}{n} \right)^{1-\frac{1}{n}+\frac{1}{q}} |\Omega|^{\frac{1}{q}}  
\end{equation}
 where $\omega_{n}$ stands for the Lebesgue measure of the unit ball. 

Regarding $K_{2}$, we note that, by \eqref{bound-gradient-F0}, 
\[
\begin{split}
\frac{|v(y)|}{|x-y|^{n-1}} X_{1}^{\frac{1}{n}} X_2^{\frac{2}{n}}  \frac{|\nabla F_{\circ}|}{F_{\circ}}   \leq & \tilde{\beta}  
\left( 
\frac{|v(y)|^{n}}{|x-y|^{(n-1)r}} X_{1} X_2^{2}  F_{\circ}^{-n}  
\right)^{\frac{1}{q}} \\
& \times  \frac{1}{|x-y|^{(n-1)\left(1-\frac{r}{q}\right)}} \left( \frac{|v(y)|^{n}}{F_{\circ}^{n}} X_{1} X_2^{2}   \right)^{\frac{1}{n}-\frac{1}{q}}.
\end{split}
\]
Integrating and applying H\"older's inequality with the exponents $q$, $\frac{n}{n-1}$ and $\frac{1}{\frac{1}{n}-\frac{1}{q}}$, and recalling \eqref{indici-01}, we get 
\[
\begin{split}
K_{2}(x)  \leq  &  \tilde{\beta} \left( 
\int_{\Omega} \frac{|v(y)|^{n}}{|x-y|^{(n-1)r}} X_{1} X_2^{2}  F_{\circ}^{-n} dy 
\right)^{\frac{1}{q}} \\
& \times  (h_{r}(x))^{1-\frac{1}{n}}  \left( \int_{\Omega} \frac{|v(y)|^{n}}{F_{\circ}^{n}} X_{1} X_2^{2} dy  \right)^{\frac{1}{n}-\frac{1}{q}}
\end{split}
\]
Therefore we can estimate
\[
\begin{split}
\|K_{2}\|_{L^{q}(\Omega)}  \leq  &  \tilde{\beta}\,  \|h_{r}(x)\|_{L^{\infty}(\Omega)}^{1-\frac{1}{n}}  \left( 
\iint_{\Omega} \frac{|v(y)|^{n}}{|x-y|^{(n-1)r}} X_{1} X_2^{2}  F_{\circ}^{-n} dy\, dx
\right)^{\frac{1}{q}} \\
& \times  \left( \int_{\Omega} \frac{|v(y)|^{n}}{F_{\circ}^{n}} X_{1} X_2^{2} dy  \right)^{\frac{1}{n}-\frac{1}{q}} \\
=  &   \tilde{\beta}\, \|h_{r}(x)\|_{L^{\infty}(\Omega)}^{1-\frac{1}{n}}  \left( 
\int_{\Omega} |v(y)|^{n}  X_{1} X_2^{2}  F_{\circ}^{-n}  h_{r}(y)  dy
\right)^{\frac{1}{q}} \\
& \times  \left( \int_{\Omega} \frac{|v(y)|^{n}}{F_{\circ}^{n}} X_{1} X_2^{2} dy  \right)^{\frac{1}{n}-\frac{1}{q}}  \\
 \leq &  \tilde{\beta}\,  \|h_{r}(x)\|_{L^{\infty}(\Omega)}^{\frac{1}{r}}  \left( \int_{\Omega} \frac{|v(y)|^{n}}{F_{\circ}^{n}} X_{1} X_2^{2} dy  \right)^{\frac{1}{n}} \\
 = &  \tilde{\beta} \, \|h_{r}(x)\|_{L^{\infty}(\Omega)}^{\frac{1}{r}}  \left( \int_{\Omega} \frac{|u(y)|^{n}}{F_{\circ}^{n}} X_{1}^{n} X_2^{2} dy  \right)^{\frac{1}{n}}.
\end{split} 
\]
Finally, by \eqref{improved-01}, we get
\begin{equation}
\label{Estim-K2}
\|K_{2}\|_{L_{q}(\Omega)}  \leq  \tilde{\beta} \, \|h_{r}(x)\|_{L^{\infty}(\Omega)}^{\frac{1}{r}}  ( n^{2}  \kappa_n I_{F,dx}[u])^{\frac{1}{n}}.
\end{equation}
A similar estimate can be obtained for $K_{1}$ by applying \eqref{link2}. Indeed, by \eqref{norm-equivalence}, we have
\[
\begin{split}
\|K_{1}\|_{L_{q}(\Omega)}  & \leq   \|h_{r}(x)\|_{L^{\infty}(\Omega)}^{\frac{1}{r}}  \left( 
\int_{\Omega} |\nabla v(y)|^{n} X_{1} ^{1-n} dy \right)^{\frac{1}{n}}  \\
& \leq \alpha^{-1}   \|h_{r}(x)\|_{L^{\infty}(\Omega)}^{\frac{1}{r}}  \left(  \int_{\Omega} F^{n}(\nabla v(y)) X_{1} ^{1-n} dy  \right)^{\frac{1}{n}} \\
& \leq  \alpha^{-1}   \|h_{r}(x)\|_{L^{\infty}(\Omega)}^{\frac{1}{r}} ( 2^{n-1}\sigma^n_F I_{F,dx}[u])^{\frac{1}{n}}.
\end{split}
\]
Using the last  estimates and \eqref{Estim-K2} in \eqref{finale-01}, we get
\begin{equation*}
\begin{split}
 \left \| u X_2^{\frac{2}{n}}\right \|_{L^{q}}  \leq  &  \frac{1}{n \omega_{n}} \left( \|K_{1}\|_{L^{q}} +  \tilde{\beta}  \frac{n+1}{n} \| K_{2}\|_{L^{q}}  \right) \\
  \leq &  \frac{1}{n \omega_{n}^{\frac{1}{n}}}   \left[  \alpha^{-1} 2^\frac{n-1}{n}\sigma_F + \tilde{\beta} \frac{n+1}{n}  n^\frac{2}{n}  \kappa_n^{\frac{1}{n}} \right]  \\ 
  & \times \left[ 1+\frac{q(n-1)}{n} \right]^{1-\frac{1}{n}+\frac{1}{q}} |\Omega|^{\frac{1}{q}}  (I_{F,dx}[u])^{\frac{1}{n}}
 \end{split} 
\end{equation*}
This gives \eqref{q-estimate-anisotropic}.
\end{proof}

We are now in position to state the main theorems. We first deal with the general case in Theorem \ref{anis-main-theorem}, while in Theorem \ref{theorem-radial} we address the case of anisotropically radial functions and we improve the result up to the optimal exponent of exponential summability.

\begin{theorem}\label{anis-main-theorem} Let $F$ be a strongly convex Finsler norm and $F_{\circ}$ its polar. Let $\Om\subset\Rn$ be a bounded domain that contains the origin and $R_\Om:=\sup_{x\in\Om}F_{\circ}(x)$. 
We have for $u\in W_{0}^{1,n}(\Om)$
\begin{equation}
\label{MT multidim Finsler}
\int_{\Om}e^{\gamma\left[|u(x)|X^{\frac{2}{n}}_2\left(\frac{F_{\circ}(x)}{R_\Om}\right)\right]^{\frac{n}{n-1}}}~\dd x
  <+\infty,
\end{equation}
for any $\gamma < C_{n,F}^{n}\,I_{F,dx}[u]^{-\frac{1}{n-1}}$, where $C_{n,F}$ is defined in \eqref{cienne} and 
\[
I_{F,\dd x}[u]:=\int_\Om F(\nabla u)^n \dd x
  -\left(\frac{n-1}{n}\right)^n\int_\Om\frac{|u|^n}{F^n_{\circ}}X^n_1\left(\frac{F_{\circ}}{R_\Om}\right)~\dd x\,.
\]
\end{theorem}

\begin{proof}
The proof closely follows \cite[Proof of Theorem 1.1]{MT}. By monotonicity, since $X_{2}<1$, it is sufficient to prove the result for $\beta=2/n$. Let $u\in W_{0}^{1,n}(\Om)$, we write
\[
\begin{split}
\int_{\Om} \sum_{k=0}^{m} &  \frac{1}{k!} \left[ c \left( u X_2^{\frac{2}{n}}\Big(\frac{F_{\circ}(x)}{R_\Om}\Big) \right)^{\frac{n}{n-1}} \right ]^{k} dx =  \int_{\Om} \sum_{k=0}^{n-1}  \frac{1}{k!} \left[ c \left( u X_2^{\frac{2}{n}} \right)^{\frac{n}{n-1}} \right ]^{k} dx  \\ 
 + &  \int_{\Om} \sum_{k=n}^{m}  \frac{1}{k!} \left[ c \left( u X_2^{\frac{2}{n}} \right)^{\frac{n}{n-1}} \right ]^{k}dx   = S_{1} + S_{2}
\end{split}
\]
Using  Proposition \ref{uX2estimate} combined with Jensen inequality we can estimate $S_{1}$ with a constant depending only on the dimension, the anisotropy and $I_{F}[u]$, indeed we have, for $k < n$,
\[
\begin{split}
\frac{1}{|\Om|}\int_{\Om}  \left( u X_2^{\frac{2}{n}}\left(\frac{F_{\circ}(x)}{R_\Om}\right) \right)^{\frac{nk}{n-1}} dx  & \leq
\left( \frac{1}{|\Om|}\int_{\Om}  \left( u X_2^{\frac{2}{n}}\left(\frac{F_{\circ}(x)}{R_\Om}\right) \right)^{\frac{n^{2}}{n-1}} dx \right)^{\frac{k}{n}}  \\
& \leq C(n,I_{F,dx}[u]). 
\end{split}
\]

To estimate $S_{2}$, we observe that, by applying \eqref{q-estimate-anisotropic} with $q=\frac{nk}{n-1}$ and $k\in \{ n,n+1,\dots\}$, we obtain
\[
\int_{\Om} \left( u X_2^{\frac{2}{n}}\Big(\frac{F_{\circ}(x)}{R_\Om}\Big) \right)^{\frac{nk}{n-1}} dx \leq C_{n}^{\frac{nk}{n-1}} |\Omega| (1+k)^{1+k}I_{F,dx}[u]^{\frac{k}{n-1}},
\]
with $C_{n}$ defined in \eqref{cienne}. 

And therefore we can write the estimate
\[
\begin{split}
\int_{\Om} \sum_{k=0}^{m} &  \frac{1}{k!} \left[ c \left( u X_2^{\frac{2}{n}}\Big(\frac{F_{\circ}(x)}{R_\Om}\Big) \right)^{\frac{n}{n-1}} \right ]^{k} dx  \leq C(n,I_{F,dx}[u])  \\ 
 + &  \sum_{k=n}^{m} \left( c\, C_{n}^{\frac{n}{n-1}}  I_{F,dx}[u]^{\frac{1}{n-1}}  \right)^{k} |\Omega| \frac{(1+k)^{1+k}}{k!}, 
\end{split}
\]
and the result follows observing that, when $m$ goes to $+\infty$, the sum on right hand side converges if $c< (C_{n}^{n}I_{F,dx}[u])^{-\frac{1}{n-1}}$.
\end{proof}

\begin{theorem}
\label{theorem-radial}
Let $F$ be a strongly convex Finsler norm and $F_{\circ}$ its polar. Let $\Om\subset\Rn$ be a bounded domain that contains the origin and $R_\Om:=\sup_{x\in\Om}F_{\circ}(x)$. 
For any $u\in W_{0}^{1,n}(\Om)$, $F_{\circ}$-radially symmetric, we have
\begin{equation}
\label{MT multidim Finsler}
\int_{\Om}e^{\bar{\gamma}\left[|u(x)|X^{\frac{1}{n}}_2\left(\frac{F_{\circ}(x)}{R_\Om}\right)\right]^{\frac{n}{n-1}}}~\dd x
  <+\infty,
\end{equation}
with 
\[
\bar{\gamma}:=n  \left( \frac{4\,\omega_{n,F}}{n\kappa_{n} J_{F}[u]}\right)^{\frac{1}{n-1}}
\]
	where  $\kappa_n:=\frac{2}{n} \big(\frac{n}{n-1}\big)^{n-2}$, $\omega_{n,F}= \mathcal{H}^{n}\big(\{F_{0}\leq 1\}\big)$ and
	\[
J_{F}[u]:=\int_\Om \Big|\nabla u\cdot\frac{x}{F_{\circ}}\Big|^n |\nabla F_{\circ}| \dd x
  -\left(\frac{n-1}{n}\right)^n\int_\Om\frac{|u|^n}{F^n_{\circ}}X^n_1\left(\frac{F_{\circ}}{R_\Om}\right)~|\nabla F_{\circ}| \dd x\,.
\]
\end{theorem}

\begin{proof}
For the anisotropically radial functions we can argue as \cite[Theorem 3.5]{BlPPs-ind}, we give the proof for read's convenience. Let $u$ be $F_{\circ}$-radially symmetric which implies that $v:=X_1^{1-1/n}u$ is radial too. 
Let $v(x)=\tilde f(F_{\circ}(x))$ with $\tilde v : [0, +\infty) \to \mathbb R$  and observe that, using \eqref{F-prop-01}, we readily have that
\begin{equation}
	\label{Eq:radial-prop-01}
	|\tilde{v}'(F_{\circ}(x))|= \Big|\nabla v(x)\cdot\frac{x}{F_{\circ}(x)}\Big|.
\end{equation}	
In the sequel we will use the same notation for $v$ and its radial expression $\tilde{v}$. Using the co-area formula and since $R_\uw=1$, by \eqref{Eq:radial-prop-01}, we get
\begin{equation}
\begin{split}\label{eq:co-area-01}
 \int_\uw & F^{2-n}_{\circ} |v|^{n-2} \Big|\nabla v\cdot\frac{x}{F_{\circ}}\Big|^2 X_1^{-1}~\dd x_{F} \\ 
 & =  \int_0^1  t^{2-n} |v(t)|^{n-2} |v'(t)|^2 X_1^{-1} (t) \int_{\{F_{\circ}=t\}}|\nabla F_{\circ}(y)| \frac{1}{|\nabla F_{\circ}(y)|} ~\dd \mathcal{H}^{n-1}(y)\, \dd t
\\ & = n \mathcal{H}^{n}(\uw) \int_0^1  t |v'(t)|^2 X_1^{-1} (t)~\dd t.
\end{split}
\end{equation}
Therefore inequality \eqref{link} with $\mu:=|\nabla F_{\circ}| \dd x$ can be rewritten as
\begin{equation*}
n\omega_{n,F} \int_{0}^{1} |v(r)|^{n-2}|v'(r)|^{2}\left( 1- \log(r) \right) r \,dr \leq \kappa_{n} J_{F}[u]
\end{equation*}
which is
\begin{equation}\label{eq-temp-01}
 \int_{0}^{1} \left|\left( |v|^{n/2}\right)'\right|^{2} \left( 1- \log(r) \right) r \,dr \leq  \frac{n\kappa_{n}J_{F}[u]}{4\omega_{n,F}}.
\end{equation}
Applying \eqref{onedimlemma stepone INEQ assist} with $g= |v|^{n/2}$, by \eqref{eq-temp-01} and raising to power $n/2$, we obtain
\begin{equation}\label{eq-02-v-est}
 |v(r)| \leq \left( \frac{n\kappa_{n}J_{F}[u]}{4\omega_{n,F}} \log X_{1}^{-1}(r) \right)^{1/n}.
\end{equation}
Now we can use \eqref{eq-02-v-est} and apply the change of variable $X_{1}^{-1}(r)=s$ that, together with \eqref{log-prop-01} and \eqref{log-prop-02}, will implies that
\begin{equation}\label{eq:est-int}
\begin{split}
\int_{\uw} e^{\bar\gamma |u|^{\frac{n}{n-1}}  X_{2}^{\frac{1}{n-1}}}  dx &   = \int_{\uw} e^{\bar\gamma |v|^{\frac{n}{n-1}}X_{1}^{-1}X_{2}^{\frac{1}{n-1}}}dx \\
& \leq n\,\omega_{n,F} \int_1^{+\infty} e^{\bar\gamma \left(  \frac{n\kappa_{n}J_{F}[u]}{4\omega_{n,F}}\right)^{\frac{1}{n-1}}  \left(\frac{\log s}{1+\log s}\right)^{\frac{1}{n-1}}  s -ns}  ds.  
\end{split}
\end{equation}
The result follows once we observe that the integral in the last term of \eqref{eq:est-int} converges if and only if
\begin{equation}
\label{estimalpha}
\bar\gamma \leq n \left( \frac{4\,\omega_{n,F}}{n\kappa_{n}J_{F}[u]}\right)^{\frac{1}{n-1}}.
\end{equation}
\end{proof}

\begin{remark}
We observe that in planar euclidean case, i.e. $n=2$ and $F(\cdot)=| \cdot |$, $\tilde{\gamma}= 4\pi$ is exactly the Moser's sharp constant for any function with $J_{F}[u]\leq1$. 
\end{remark}

\end{document}